\documentclass[12pt]{article}
\usepackage{amsmath}
\usepackage{amssymb}
\usepackage{amsfonts}
\usepackage{amsthm}
\RequirePackage[russian, english]{babel}

\usepackage{amsbib}

\newtheorem{theorem}{Theorem}
\newtheorem{lemma}{Lemma}
\newtheorem{cor}{Corollary}
\newtheorem{remark}{Remark}

\newtheorem*{Prop}{Proposition}

\newcommand{\RR}{{\Bbb R}}
\newcommand{\ZZ}{{\Bbb Z}}
\newcommand{\NN}{{\Bbb N}}

\newcommand{\eps}{\epsilon}
\newcommand{\cA}{{\mathcal A}}

\newcommand{\und}[1]{\underset{#1}}

\newcommand{\Gcd}{\text{\rm gcd } }

\title{On the cardinality of the exception set in Littlewood's discrete conjecture}
\author{A.A. Illarionov\footnote{The research leading to these results has received funding from the Basic Research Program at HSE University (HSE-BR-2025-84)} (HSE University, Moscow)}

\begin{document}
\maketitle
\abstract{
Let $p$ be a prime, $d\ge 2$, $H\in [1,p)$, and $\log\log p = o(\log H)$. We prove that
$$
    \min_{1\le n \le H} n \left\|\frac{a_1 n}{p}\right\|\ldots \left\|\frac{a_d n}{p}\right\| \approx \frac{1}{(\log p)^{d-1} \log H}
$$
for "almost all" $a \in (\ZZ / p\ZZ)^d$.
}


{\bf AMS:} 11J13, 11J70, 11B50

\section{Introduction}

The famous Littlewood's conjecture (LC) is as follows: {\it if $d\ge 2$, $\alpha \in \RR^d$, then
$$
    \liminf_{n\to \infty} n \|\alpha_1 n\| \ldots \|\alpha_d n\| = 0.
$$
}

There are many results concerning  with  LC (see \cite{Einsiedler, Cassels, Pollington, Lindenstrauss, BugeaudMosh, Badziahin2013,    Fregoli} and the references therein). The strongest result is
\begin{theorem}[\cite{Einsiedler}]
If $d\ge 2$, the set
$$
 \{\alpha \in \RR^d:    \liminf_{n\to \infty} n \|\alpha_1 n\| \ldots \|\alpha_d n\| >0\}
$$
has zero Hausdorff dimension.
\label{th1}
\end{theorem}

Shkredov \cite{Shkredov} considered some discrete analogue of LC and Theorem \ref{th1} that arise in the theory of numerical integration (see \cite{KorobovMono, Korobov1967, HuaLoo}).

Let $p$ be a prime and $a=(a_1,\ldots, a_d) \in \ZZ_p^d$. Here and below $\ZZ_p = \ZZ/ p\ZZ$. For any $H\in [2,p)$ define
$$
    q_p(a,H) = \min_{1\le n \le H} n \left\|\frac{a_1 n}{p}\right\|\ldots \left\|\frac{a_d n}{p}\right\| .
$$
\begin{Prop}[{\cite[Proposition 5]{Shkredov}}] Let $d\ge 2$, $f:\NN \to [1,+\infty)$, and $f(p) = o (p^{(d+1)/2d})$. Suppose that for any $\eps>0$ there are infinitely many primes p for which
$$
    \max_{a\in \ZZ_p^d} q_p (a, f(p)) < \eps.
$$
Then LC take place in $\RR^d$.
\end{Prop}

It's easy to prove that (see \cite{Shkredov, Cusick})
$$
    \max_{a\in \ZZ_p^d} q_p (a,p) \und{d}\gg \frac{1}{\log^d p}.
$$

Consider the set of real numbers $F_M$, having all partial quotients bounded by $M$. It is well-known, that this set has zero Lebesgue measure. Let $\omega_M$  be the Hausdorff dimension of $F_M$.
Hensley \cite{Hensley} proved that
$$
    \omega_M = 1 - \frac{6}{\pi^2 M} - \frac{72 \log M}{\pi^4 M^2} + O\left(\frac{1}{M^2}\right).
$$
The following result can be seen as a discrete analogue of Theorem \ref{th1}.

\begin{theorem}[{\cite[Theorem 3]{Shkredov}}] Let $d\ge 1$, $p$ be a prime, $\eps \in (0,1/2)$, $M \gg d^2 / \eps$, $M\gg 1$. Then
$$
    \frac{1}{p^d} \# \left\{a \in \ZZ_p^d: \; q_p(a,p) \ge \frac{1}{M}\right\} \ll \frac{M^{d(d+1)}}{ p^{d - (1-\omega_M) d(d+1) + c \eps d^2 /M }},
$$
where $c$ is a positive constant.
\label{th2}
\end{theorem}

The proof is based on some additive-combinatorial results and estimates of special character sums.

The bound from Theorem \ref{th2} is not trivial for small $M$, for example, $M \ll \log p (\log\log p)^{-1}$.

Let $H\in [2,p)$, $1\le \lambda \le p^d $. It is easy to check that (see Section~\ref{r2})
\begin{equation}
\frac{1}{p^d} \# \left\{a \in \ZZ_p^d: \; q_p(a,H) \le \frac{1}{\lambda (\log p)^{d-1} \log H}\right\} \und{d}\ll \frac{1}{\lambda}.
\label{1.1}
\end{equation}
The main result of this paper is the following.
\begin{theorem}
Let $d\ge 2$, $p$ be a prime, $H\in [2,p]$, $1\le \lambda \le (\log p)^{d-1} \log H$, and $\log\log p = o(\log H)$. Then
\begin{equation}
\frac{1}{p^d} \# \left\{a \in \ZZ_p^d: \; q_p(a,H) > \frac{\lambda }{(\log p)^{d-1} \log H}\right\} \und{d}\ll \frac{1}{\lambda}.
\label{1.2}
\end{equation}
\label{th3}
\end{theorem}
\begin{remark}
Let $H=p$. Then Theorem \ref{th2} stronger than Theorem \ref{th1} if $\lambda$ is a small, for example, $\lambda < (\log p)^{d-2} \log\log p$.
\end{remark}

From \eqref{1.1}, \eqref{1.2} it follows that
$$
    q_p(a,H) \approx \frac{1}{(\log p)^{d-1} \log H}
$$
for "almost all" $a\in \ZZ_p^d$.

We use the following notation.
If $x\in \RR$, then $\|x\|$ is the distance to the nearest integer.
The expression $f(x) \ll g(x)$ (or $f(x) = O(g(x))$) for $x \in  X$ means that there exists a positive absolute constant $C$ such that $|f(x)| \le  C g(x)$
for all $x \in  X$. If $C$ depends on a parameter $\theta$, then we write $f(x) \und{\theta}\ll g(x)$
(or $f(x) = O_\theta(g(x)))$. The notation $f \asymp g$ means that $f \ll g \ll f$.
By $\# M$ we denote the cardinality of the set $M$.

\section{Proof outline}
\label{r2}

For any $a\in \ZZ_p^d$ we define the lattice
$$
    \Lambda(a) = \Lambda(a,p) = \{(x_0,\ldots,x_d): \; a_i x_0 \equiv x_i \pmod p, \; 1\le i\le d\}.
$$
Then
\begin{equation}
    q_p(a,H) = \frac{1}{p^d} \min\{x_0 |x_1|\ldots |x_d|: \;  x\in \Lambda(a), \; 1 \le x_0 \le H\}.
\label{6.1}
\end{equation}
Take any set $\Omega \subset \ZZ^{d+1}$ such that
$$
    1\le x_0 |x_1|\ldots |x_d| <p
$$
for all $x=(x_0,\ldots,x_d) \in \Omega$. Put
$$
    S(a;\Omega) = \# (\Omega \cap \Lambda(a)) = \sum_{x\in \Omega} \prod_{i=1}^d \delta_p(a_i x_0 - x_i).
$$
Here and below
$$
    \delta_m(b) = \begin{cases} 1 & \mbox{ if } b\equiv 0 \pmod m \\ 0 & \mbox{ otherwise }\end{cases} \qquad \mbox{for } m\in \NN,\; b\in \ZZ.
$$

Considering $a$ as a random variable uniformly distributed on $\ZZ_p^d$, we define the expectation $\mu_\Omega$ and variance $\sigma_\Omega$ of $S(a;\Omega)$ in the standart way:
\begin{gather*}
    \mu_\Omega = \frac{1}{p^d} \sum_{a\in \ZZ_p^d} S(a;\Omega),
    \\
    \sigma^2_\Omega = \frac{1}{p^d} \sum_{a\in \ZZ_p^d} (S(a;\Omega)-\mu_\Omega)^2 = \frac{1}{p^d} \sum_{a\in \ZZ_p^d} S(a;\Omega)^2 - \mu_\Omega^2.
\end{gather*}
Let
\begin{multline*}
    W_\Omega =\{(x,y)\in \Omega\times \Omega : \; x,y \mbox{ are linearly dependent over $\ZZ_p$}\} ={}
    \\
    {}= \{(x,y)\in \Omega\times \Omega : \; x_0 y_i \equiv y_0 x_i \pmod p, \; 1\le i \le d \}.
\end{multline*}

\begin{lemma}
The relations
$$
    \mu_\Omega = \frac{\# \Omega}{p^d}, \quad \sigma_\Omega^2 = \frac{\# W_\Omega}{p^d} - \mu_\Omega^2
$$
hold.
\label{l6.1}
\end{lemma}
\begin{proof}
Clearly,
$$
    p^d \mu_\Omega = \sum_{x\in \Omega} \sum_{a\in \ZZ_p^d}\prod_{i=1}^d \delta_p(a_i x_0 - x_i) = \sum_{x\in \Omega} 1 = \# \Omega.
$$
In addition,
$$
    \sum_{a\in \ZZ_p^d} \prod_{i=1}^d\delta_p(a_i x_0 - x_i)\delta_p(a_i y_0 - y_i) = \begin{cases} 1 & \mbox{ if $x,y$ are linearly dependent over $\ZZ_p$} \\ 0 & \mbox{ otherwise }\end{cases}.
$$
Therefore,
$$
    \sigma^2_\Omega + \mu_\Omega^2 = \frac{1}{p^d}\sum_{a\in \ZZ_p^d} S(a;\Omega)^2 =  \frac{1}{p^d}\sum_{x,y\in \Omega}\sum_{a\in \ZZ_p^d} \prod_{i=1}^d\delta_p(a_i x_0 - x_i)\delta_p(a_i y_0 - y_i) =
    \frac{\# W_\Omega}{p^d}.
$$
This completes the proof of the lemma.
\end{proof}

According to Chebyshev's inequality
$$
    \forall t>0 \quad \frac{1}{p^d} \# \left\{a\in \ZZ_p^d: |S(a;\Omega)-\mu_\Omega|\ge t \mu_\Omega\right\} \le \frac{\sigma^2_\Omega}{t^2 \mu^2_\Omega}.
$$
Choosing $t=1/2$ we obtain
$$
    \frac{1}{p^d} \# \left\{a\in \ZZ_p^d: S(a;\Omega)=0\right\}  \le \frac{4\sigma^2_\Omega}{\mu^2_\Omega}.
$$
Suppose the set $\Omega$ satisfies the conditions
\begin{equation}
     1\le x_0 \le H, \quad x_0 |x_1|\ldots |x_d| \le \frac{\lambda p^d}{ (\log p)^{d-1} \log H}.
\label{6.2}
\end{equation}
for all $x \in \Omega$. Then from \eqref{6.1} it follows that $S(a;\Omega) = 0$ if $q_p(a,H) > \lambda ( (\log p)^{d-1} \log H)^{-1}$. Hence
\begin{equation}
    \frac{1}{p^d} \#\left\{a\in \ZZ_p^d: \; q_p(a,H) > \frac{\lambda}{(\log p)^{d-1} \log H} \right\} \le \frac{4\sigma^2_\Omega}{\mu^2_\Omega}.
\label{6.3}
\end{equation}

In order to prove Theorem \ref{th3}, it remains to obtain an upper bound for $\sigma_\Omega$. For this, in turn, it suffices to get an asymptotical formula for $\# W_\Omega$.

\begin{lemma}
Let $H\in [2,p)$, $1\le \lambda \le p^d $. Then relation \eqref{1.1} holds.
\label{l6.2}
\end{lemma}
\begin{proof}
Let $\Omega$ be the set of all $x\in \ZZ^{d+1}$ such that
$$
 1 \le x_{0} < H, \quad 1 \le x_0|x_1| \ldots |x_d| < \frac{p^d}{\lambda (\log p)^{d-1} \log H}.
$$
From \eqref{6.1} it follows that $S(a;\Omega) \ge 1$ for any $a\in \ZZ_p^*$ such that $q_p(a;H) \le  (\lambda(\log p)^{d-1} \log H)^{-1}$. Hence
\begin{multline*}
    \# \left\{a\in \ZZ_p^d: \; q_p(a,H) \le \frac{1}{\lambda (\log p)^{d-1} \log H} \right\} \le \sum_{a\in \ZZ_p^d} S(a;\Omega) = p^d \mu_\Omega = \#\Omega  \le {}
    \\
    {}\le
    \sum_{1\le |x_1|, \ldots, |x_{d-1}| < p^d} \sum_{1\le x_0 \le H} \frac{p^d}{\lambda (\log p)^{d-1} (\log H) |x_1|\ldots |x_{d-1}| x_0 }
    \und{d}\ll \frac{p^d}{\lambda}.
\end{multline*}

This completes the proof of the lemma.
\end{proof}

\section{Number of solutions of some congruences}
\label{r3}

The aim of this section is to prove Corollary \ref{cor3.1}.

Take any $A,B \in \NN$ and $x_0,y_0 \in [1,+\infty)$. Let $f(x_0,y_0; A,B)$ be the number of pairs $(u,v) \in \NN^2$ such that
\begin{equation}
    x_0 v \equiv u y_0 \pmod p, \quad A \le u < 2A, \quad B \le v < 2B.
    \label{3.0}
\end{equation}

\begin{lemma}
Suppose that
$$
  1\le  x_0 \log x_0 \ll A \le p, \quad 1\le  y_0 \log y_0 \ll B\le p; \quad r = \Gcd(x_0, y_0), \quad M= \max\{x_0 B, y_0 A\}.
$$
Then
\begin{eqnarray}
f(x_0,y_0;A,B) &=& \frac{AB}{p} + O\left(\frac{AB}{M}r + \frac{M \log (x_0+1)}{rp} + \log (x_0+1) \right),
\label{3.1}
\\
f(x_0,y_0) &\ll& \frac{AB r}{\min\{M, rp\}} = \max\left\{\frac{AB}{p}, \frac{AB}{M} r\right\}.
\label{3.2}
\end{eqnarray}
\label{l3.1}
\end{lemma}
\begin{remark}
Relation \eqref{3.1} is an asymptotical formula if $p=o(M)$, $M\log p = o(AB)$.
\end{remark}
\begin{proof}
Without loss of generality, we can assume that $r=1$ and $M=x_0 B$, i.e. $x_0 B \ge y_0 A$.

1. Let us prove \eqref{3.1}. If $x_0 v\equiv  y_0 u \pmod p$, then there exists an integer $k$ such that
$$
    x_0 v = y_0 u  + kp.
$$
Hence $f(x_0,y_0;A,B)$ is equal to the number of $(u,k)\in \NN\times \ZZ$ for which
$$
    -u y_0 \equiv kp\pmod{x_0}, \quad A \le u< 2 A, \quad \frac{x_0 B - u y_0}{p} \le k < \frac{2x_0 B - u y_0}{p}.
$$
For any $t\in \RR$ we define $e(t) = e^{2\pi i t}$. We will use the following well-known relations
$$
    \sum_{n\in \ZZ_m} e(bn/m) = m\delta_m(b), \quad \sum_{l=1}^{m-1} \left|\sum_{C\le n< C+\Delta} e(cn/m)\right| < m\log m,
$$
where $m\in \NN$, $b\in \ZZ$, $c\in \ZZ$, $\Gcd(c,m)=1$, $C\in \RR$, $\Delta \in (0,+\infty)$.

Let $J= J(u) = (x_0 B - u y_0)/p$. Then
$$
f(x_0,y_0;A,B) = \sum_{A\le u < 2A} \sum_{J\le k \le J + x_0 B p^{-1}} \delta_{x_0} (u y_0 + kp) =  \frac{1}{x_0} \sum_{u, k} \sum_{n=0}^{x_0-1} e\left(\frac{u y_0+kp}{x_0} n\right)  = S+ R,
$$
where
\begin{eqnarray*}
S &=& \frac{1}{x_0} \sum_{A\le u< 2 A} \sum_{J\le k < J + x_0 B p^{-1}} 1 = \frac{1}{x_0} \left(A + O(1)\right)\left(\frac{x_0 B}{p} + O(1)\right) =
\frac{AB}{p} + O\left( \frac{A }{x_0} + \frac{x_0 B}{p} + 1\right),
\\
|R| &=& \frac{1}{x_0} \left|\sum_{A\le u< 2 A} \sum_{J\le k < J + x_0 B p^{-1}} \sum_{n=0}^{x_0-1} e\left(\frac{u y_0+kp}{x_0} n\right) \right| \le{}
\\
&\le&
\frac{1}{x_0} \sum_{n=1}^{x_0-1} \left|\sum_{u} e(u y_0 n/ x_0)\right| \cdot \left|\sum_{k} e(k pn/x_0)\right| \ll
\frac{1}{x_0} \left(\frac{x_0 B}{p}+1\right) x_0 \log x_0 = \left(\frac{x_0 B}{p}+1\right) \log x_0.
\end{eqnarray*}
Thus,
$$
    f(x_0,y_0; A,B) = \frac{AB}{p} + O\left( \frac{A }{x_0} + \frac{x_0 B}{p} \log x_0 + \log x_0\right).
$$
It remains to see that
$$
    \frac{A }{x_0} = \frac{AB}{M}, \quad \frac{x_0 B}{p} \log x_0 = \frac{M \log x_0}{p}.
$$
This completes the proof of  \eqref{3.1}.

2. Let's prove bound \eqref{3.2}. If $2M \ge p$, then
$$
   \frac{AB}{M} \ll \frac{AB}{p}, \quad \log x_0 \le \frac{M \log x_0}{p} \ll \frac{AB}{p} \quad \Longrightarrow \quad f(x_0,y_0;A,B)    \ll \frac{AB}{p}.
$$
Suppose that $2M < p$. Take any pair $(u,v)$ satisfying \eqref{3.0}. Then $x_0 v = y_0 u + kp$, where $k\in \ZZ$. If $k\neq 0$, then
$$
    p \le p|k| = |x_0 v - y_0 u| \le 2M < p.
$$
Therefore, $k=0$, i.e. $x_0 v = y_0 u$. Using the condition $\Gcd (x_0,y_0) = 1$, we have $u = t x_0$, $v = t y_0$, where $t\in \NN$.
Since $u = t x_0  \le 2 A$, then $t \le 2 A / x_0$. Hence
$$
    f(x_0,y_0;A,B) \le 2 A / x_0 = 2AB/ M.
$$
This completes the proof of the lemma.
 \end{proof}

Take any $P=(P_0,\ldots,P_d)$, $Q=(Q_0,\ldots,Q_d) \in [1,+\infty)^{d+1}$. Let $\cA (P,Q)$ be the number of collections $(x,y)\in \NN^{d+1}\times \NN^{d+1}$ such that
\begin{equation}
\begin{array}{c}
x_0 y_j \equiv x_j y_0 \pmod p, \quad j = 1,\ldots,d,
\\
P_j \le x_j < 2 P_j, \quad Q_j \le y_j < 2 Q_j, \quad j = 0,\ldots,d.
\end{array}
\label{3.4}
\end{equation}
Let $\cA' (P,Q)$ be the number of collections $(x,y)\in \NN^{d+1}\times \NN^{d+1}$ for which
$$
\begin{array}{c}
x_0 y_j = x_j y_0, \quad j = 1,\ldots,d,
\\
P_j \le x_j < 2 P_j, \quad Q_j \le y_j < 2 Q_j, \quad j = 0,\ldots,d.
\end{array}
$$

\begin{lemma}
Suppose that
$$
1\le P_0 \log p \und{d}\ll P_1 \le \ldots \le P_{d} \und{d}\ll p, \quad 1\le Q_0 \log p \und{d}\ll Q_1 \le \ldots \le Q_{d} \und{d}\ll p.
$$
Put $M_j = \max\{P_0 Q_j, P_j Q_0\}$, $\hat P = P_0P_1 \ldots P_d$, $\hat Q = Q_0 Q_1 \ldots Q_d$. Then
\begin{gather*}
    \cA(P,Q) = \frac{\hat P \hat Q}{p^d}
    \left(1+ O(R)\right)  + O(\cA'(P,Q)),
    \\
    R = \frac{M_1 \log p}{P_1 Q_1} + \frac{1}{\min\{P_0,Q_0\}} + \frac{p\log p}{M_1}+ \sum_{l=2}^{d-1} \frac{p^{l}\min\{P_0,Q_0\}^{l-1}}{M_1 \ldots M_{l}}.
\end{gather*}
\label{l3.2}
\end{lemma}
\begin{proof}
Clearly,
$$
    \cA(P,Q) = \sum_{r\ge 1} \sum_{(x_0,y_0)=r} F(x_0,y_0), \quad F(x_0,y_0) = \prod_{j=1}^d f(x_0,y_0;P_j,Q_j)
$$
where $P_0 \le x_0 < 2P_0$, $Q_0 \le y_0 < 2 Q_0$.
Let's split the sum into two parts
$$
    \cA(P,Q) = S + R_1, \quad S= \sum_{1\le r\le 4M_1 p^{-1}}\sum_{(x_0,y_0)=r} F(x_0,y_0), \quad R_1 =  \sum_{r> 4M_1 p^{-1}}\sum_{(x_0,y_0)=r} F(x_0,y_0).
$$

1. In first we bound the $R_1$. Take any pair $(x,y)$ satysfying \eqref{3.4}. Put $r= \Gcd (x_0,y_0)$.
If  $4 M_{d} < rp$, then
$$
    M_1 \le M_2 \le \ldots \le M_{d} < \frac{r p}{4}.
$$
Since $x_0 y_j \equiv x_j y_0 \pmod p$, we have
$$
    x_0 y_j = x_j y_0 + k_j p, \quad k_j \equiv 0 \pmod r, \quad |k_j| p = |x_0 y_j - x_j y_0| \le 4 M_j < rp.
$$
Therefore, $k_j = 0$, $1\le j\le d$. Hence $x,y$ are linearly dependent over $\RR$. Thus,
$$
    R_1 = \sum_{l=1}^{d-1} \sum_{4M_l p^{-1} < r \le 4 M_{l+1} p^{-1}} \sum_{(x_0,y_0) = r} F(x_0,y_0) + O(\cA' (P,Q)).
$$
Using Lemma \ref{l3.1}, we get
$$
f(x_0,y_0; P_j,Q_j) \ll \frac{P_j Q_j}{M_j}r \mbox{ if } j\le l; \qquad f(x_0,y_0; P_j,Q_j) \ll \frac{P_jQ_j}{p}  \mbox{ if } j> l.
$$
Hence
$$
    F(x_0,y_0) \und{d}\ll \frac{P_1\ldots P_{d} Q_1\ldots Q_{d}}{p^{d-l}} \frac{r^{l}}{M_1 \ldots M_{l}} = \frac{\hat P \hat Q}{P_0 Q_0 p^{d-l}} \frac{r^{l}}{M_1 \ldots M_{l}}.
$$
Put $m= \min\{2P_0,2Q_0\}$. Then $r=\Gcd (x_0,y_0) \le m$,
\begin{multline*}
 \sum_{l=1}^{d-1} \sum_{4M_l p^{-1} < r \le 4 M_{l+1} p^{-1}} \sum_{(x_0,y_0) = r} F(x_0,y_0) \und{d}\ll
 \\
 {}\und{d}\ll
 \sum_{l=1}^{d-1} \sum_{1\le r \le m} \sum_{(x_0,y_0)=r} \frac{\hat P \hat Q}{ P_0 Q_0 p^{d-l}} \frac{r^{l}}{M_1 \ldots M_l} \und{d}\ll
 \frac{\hat P \hat Q}{p^d} \sum_{l=1}^{d-1} \sum_{1\le r\le m} \frac{p^{l} r^{l-2}}{M_1 \ldots M_l} \ll{}
 \\
 {}\ll
 \frac{\hat P \hat Q}{p^d} \left(\frac{p}{M_1}\log m + \sum_{l=2}^{d-1}\frac{p^{l} m^{l-1}}{M_1 \ldots M_l}\right) \ll \frac{\hat P \hat Q}{p^d} R.
\end{multline*}
Thus, we have
$$
    R_1 \und{d}\ll \frac{\hat P \hat Q}{p^d} R + \cA'(P,Q).
$$

2. It remains to consider $S$. Using Lemma \ref{l3.1}, we obtain
\begin{multline*}
    S = \sum_{(x_0,y_0) \le 4M_1 p^{-1}} f(x_0,y_0; P_1,Q_1)\ldots f(x_0,y_0; P_d,Q_d) ={}
    \\
    {}=
    \sum_{(x_0,y_0) \le 4M_1 p^{-1}} \prod_{j=1}^d \left( \frac{P_j Q_j}{p} + O\left( \frac{P_j Q_j r}{M_j} + \frac{M_j \log p}{pr} \right)\right),
\end{multline*}
where $r = (x_0,y_0)$. Since $rp\ll M_j$, $M_j \ll P_j Q_j (\log p)^{-1}$, then
$$
    \frac{P_j Q_j r}{M_j} + \frac{M_j \log p}{pr} \ll \frac{P_jQ_j}{p} \quad \Longrightarrow
    \quad
    S = \sum_{(x_0,y_0) \le 4M_1 p^{-1}} \left(\frac{\hat P \hat Q}{ P_0 Q_0p^d} +O(R_2(x_0,y_0))\right),
$$
where
\begin{gather*}
    R_2(x_0,y_0) = \sum_{j=1}^{d} \left(\frac{P_j Q_j r}{M_j} + \frac{M_j \log p}{pr}\right) \prod_{1\le i \le d, \atop i\neq j} \frac{P_i Q_i}{p} = \frac{\hat P \hat Q}{P_0 Q_0 p^{d-1}} R_3(x_0,y_0),
    \\
    R_3(x_0,y_0) = \sum_{j=1}^d \left(\frac{r}{ M_j} + \frac{M_j \log p}{P_j Q_j pr}\right) \ll \frac{r}{M_1} + \frac{M_1 \log p}{P_1 Q_1 r p}.
\end{gather*}
We took into account that $M_j (P_jQ_j)^{-1} \le M_1 (P_1 Q_1)^{-1}$. Thus,
$$
    S = \frac{\hat P\hat Q}{P_0 Q_0 p^d} \sum_{(x_0,y_0) \le 4 M_1 p^{-1}} \left(1+ O_d \left(\frac{rp}{M_1} + \frac{M_1 \log p}{r P_1 Q_1}\right)\right).
$$
Since
$$
    \sum_{(x_0,y_0) \le 4M_1 p^{-1}} 1 = P_0Q_0\left(1 + O_d\left(\frac{1}{m} + \frac{p}{M_1}\right)\right), \quad m = \min\{2P_0,2Q_0\}
$$
we conclude
$$
    S = \frac{\hat P \hat Q}{p^d} \left(1 + O\left(\frac{1}{m} + \frac{p}{M_1} + R_3 \right)\right),
$$
where
\begin{multline*}
R_3 = \frac{1}{P_0 Q_0} \sum_{r \le 4 M_1 p^{-1}} \sum_{(x_0,y_0)=r} \left(\frac{rp}{M_1} + \frac{ M_1 \log p}{P_1 Q_1 r }\right) \ll{}
\\
{}\ll
 \sum_{r \le 4 M_1 p^{-1}} \left(\frac{p}{rM_1} + \frac{M_1 \log p}{P_1 Q_1 r^3 }\right)\ll
\left( \frac{p\log p}{M_1} + \frac{M_1 \log p}{P_1 Q_1 }\right) \ll R.
\end{multline*}
Hence
$$
    S = \frac{\hat P \hat Q}{p^d} \left(1 + O(R)\right).
$$
This completes the proof of the lemma.
\end{proof}

\begin{cor}
Let $d\ge 2$, $1\le T \le \log^d p$,
\begin{gather*}
P_0\ldots P_{d} \und{d}\gg \frac{p^d}{T}, \quad \frac{\log^{d+1} p}{T} P_0 \und{d}\ll P_1 \le P_2 \le \ldots \le P_{d} \und{d}\ll \frac{p}{\log^{d+1} p},
\quad P_0 \und{d}\gg \frac{\log^{d+1} p}{T},
\\
Q_0\ldots Q_{d} \und{d}\gg \frac{p^d}{T}, \quad \frac{\log^{d+1} p}{T} Q_0 \und{d}\ll Q_1 \le Q_2 \le \ldots \le Q_{d} \und{d}\ll \frac{p}{\log^{d+1} p},
\quad Q_0 \und{d}\gg \frac{\log^{d+1} p}{T}
\end{gather*}
Then
$$
    \cA(P,Q) = \frac{P_0\ldots P_{d} Q_0 \ldots Q_{d}}{p^d} \left(1 + O_d\left(\frac{T}{\log^{d} p}\right)\right) +  O_d(\cA'(P,Q)).
$$
\label{cor3.1}
\end{cor}
\begin{proof}
According to Lemma \ref{l3.2}, it suffices to prove that
$$
    \frac{M_1 \log p}{P_1 Q_1} + \frac{p\log p}{M_1} + \frac{1}{\min\{P_0,Q_0\}} + \sum_{l=2}^{d-1} \frac{p^{l}\min\{P_0,Q_0\}^{l-1}}{M_1 \ldots M_{l}} \und{d}\ll \frac{T}{\log^{d} p}.
$$
From conditions it follows that
\begin{gather*}
    \frac{1}{\min\{P_0,Q_0\}} \und{d}\ll \frac{T}{\log^{d+1} p},
\\
\forall i\neq j \quad
\left(\frac{p^d}{T}\right)^2 \ll P_0\ldots P_{d} Q_0\ldots Q_{d} \und{d}\ll \left(\frac{p}{\log^{d+1} p}\right)^{2d-2} P_iP_j Q_i Q_j
\\
\Longrightarrow \quad P_iQ_j P_j Q_i \und{d}\gg \frac{p^2}{T^2} (\log p)^{2(d+1)(d-1)}.
\end{gather*}
Choosing $i=0$, we have
$$
    M_j \und{d}\gg
    \frac{p (\log p)^{(d+1)(d-1)}}{T}, \quad
    \frac{p\log p}{M_1} \ll
    \frac{T}{(\log p)^{(d+1)(d-1)-1}} \le \frac{T}{(\log p)^{d+4}}.
$$
Since $M_1 \und{d}\ll P_1 Q_1 T (\log p)^{-d-1}$, then
$$
    \frac{M_1 \log p}{P_1 Q_1} \und{d}\ll \frac{T}{(\log p)^{d}}.
$$

It remains to show that
$$
  \frac{p^{l}\min\{P_0,Q_0\}^{l-1}}{M_1 \ldots M_{l}} \und{d}\ll \frac{T}{\log^{d} p}, \quad l=2,\ldots, d-1.
$$

Without loss of generality, we can assume that $P_0\le Q_0$. Then $M_j \ge  Q_0 P_j \ge P_0 P_j$, $j\ge 1$.
Take any $l\in \{2,\ldots, d-1\}$. Using conditions of the corollary, we have
$$
    \frac{p^d}{T} \und{d}\ll (P_0\ldots P_{l})(P_{l+1}\ldots P_{d}) \und{d}\ll P_0\ldots P_{l} \frac{p^{d-l}}{(\log p)^{(d+1)(d-l)}}\quad \Longrightarrow \quad
    P_0\ldots P_l \und{d}\gg \frac{p^{l}}{T} (\log p)^{(d+1) (d-l)}.
$$
Therefore,
$$
     \frac{p^{l}\min\{P_0,Q_0\}^{l-1}}{M_1 \ldots M_{l}}  \le
    \frac{p^{l} P_0^{l-1}}{(P_0 P_1)\ldots (P_0 P_l)} = \frac{p^{l}}{P_0\ldots P_l} \ll \frac{T}{(\log p)^{(d+1)(d-l)}} \le \frac{T}{\log^{d+1} p}.
$$
This completes the proof of the corollary.
\end{proof}

\section{Proof of Theorem \ref{th3}}
\label{r4}

Take any $T\in [1,\log^d p]$, $H\in[2,p]$. Let $K=K(T,H,p)$ be the set of collections $k=(k_0,\ldots,k_{d}) \in \ZZ^{d+1}_+$ ($\ZZ_+ = \NN\cap \{0\}$) such that
\begin{gather*}
\log_2\left(\frac{(\log p)^{d+1}}{T}\right) \le k_0 \le k_1 - \log_2\left(\frac{(\log p)^{d+1}}{T}\right), \quad k_0 \le \log_2 H-1,
\\
\quad k_1 \le k_2 \le \ldots \le k_{d} \le \log_2 \left(\frac{p}{\log^{d+1} p}\right),
\\
\log_2(p^d/T) -2d  \le k_0+\ldots + k_{d} \le \log_2(p^d/T).
\end{gather*}

\begin{remark}
\rm If $k,n\in K$, then the collections $P = (2^{k_0},\ldots,2^{k_{d}})$, $Q = (2^{n_0},\ldots,2^{n_{d}})$ satisfy the conditions of Corollary \ref{cor3.1}.
\label{rem4.1}
\end{remark}

Define
\begin{gather*}
    \Pi_k = \left\{(x_0,\ldots,x_d)\in \ZZ^{d+1}:\; 2^{k_i}\le x_i < 2^{k_i+1}, \; 0\le i\le d\right\},
    \\
    \Omega = \Omega(T,H,p) = \bigcup_{k\in K} \Pi_k.
\\
   S(a) = \sum_{x\in \Omega \cap \Lambda(a)}1 = \sum_{x\in \Omega} \prod_{j=1}^{d} \delta_p(a_j x_{0} - x_j) \quad \mbox{for } a\in \ZZ_p^d
\\
    \mu = \frac{1}{p^d} \sum_{a\in \ZZ_p^d} S(a), \quad \sigma^2 = \frac{1}{p^d} \sum_{a\in \ZZ_p^d} (S(a)-\mu)^2 = \frac{1}{p^d} \sum_{a\in \ZZ_p^d} S(a)^2 - \mu^2.
\end{gather*}

\begin{lemma}
Let $H\in [2,p]$,  $1\le T \le (\log p)^{d-1} \log H$, and $\log\log p = o(\log H)$; then
\begin{gather}
    \mu = C \frac{(\log p)^{d-1} \log H}{T} + O_d\left( \frac{(\log p)^{d-1} \log\log p}{T}\right),
    \label{4.a}
    \\
    \sigma \und{d}\ll \mu.
\end{gather}
where $C$ is a positive constant depending only on $d$.
\label{l4.1}
\end{lemma}
\begin{proof}
From Lemma \ref{l6.1} it follows that
\begin{multline*}
\mu = \frac{\# \Omega}{p^d} =
\frac{1}{p^d} \sum_{k\in K} 2^{k_0+\ldots k_{d}} ={}
\\
{}=
\frac{1}{p^d} \sum_{1\le k_0 \le \ldots \le k_{d} \le \log_2 p, \; k_0\le \log_2 H -1, \atop \log_2(p^d/T) -2d \le  k_0+\ldots + k_{d}\le \log_2(p^d/T)}  2^{k_0+\ldots k_{d}} +
O_d\left( \frac{\log^{d-1} p}{T} \log\log p\right) ={}
\\
{}= C \frac{(\log p)^{d-1} \log H}{T} + O_d\left( \frac{\log^{d-1} p}{T} \log\log p\right).
\end{multline*}
In addition,
$$
    \sigma^2+ \mu^2 = \frac{\# W_\Omega}{p^d}  = \frac{1}{p^d} \sum_{k,n\in K} \cA(2^k,2^n),
$$
where $2^k=(2^{k_0},\ldots, 2^{k_d})$, $2^n=(2^{n_0},\ldots, 2^{n_d})$. Using Remark \ref{rem4.1} and Corollary \ref{cor3.1}, we get
\begin{multline*}
    \frac{1}{p^d} \sum_{k,n\in K} \cA(2^k,2^n) = \frac{1}{p^d}\sum_{k,n\in K} \left(\frac{2^{k_0+\ldots + k_d}2^{n_0+\ldots + n_d}}{p^d} \left(1 + O_d\left(\frac{T}{\log^d p}\right)\right) + O(\cA'(2^k,2^n)\right)
    ={}
    \\
    {}=
     \mu^2\left(1 + O_d\left(\frac{T}{\log^d p}\right)\right) + O(R_0), \qquad R_0 = \frac{1}{p^d} \sum_{k,n\in K} \cA'(2^k,2^n).
\end{multline*}
Since $\mu \und{d}\ll (\log p)^d/ T$, it follows that
$$
    \sigma^2 = \frac{1}{p^d} \sum_{k,n\in K} \cA(2^k,2^n) - \mu^2  \und{d}\ll  \frac{T \mu^2 }{\log^d p} + R_0 \und{d}\ll \mu + R_0.
$$
It remains to prove that $R_0 = O_d(\mu)$. If $x,y$ are linearly dependent over $\RR$, and $2^{k_j} \le x_j < 2^{k_j + 1}$, $2^{n_j} \le y_j < 2^{n_j + 1}$, $k,n\in K(T)$, then
\begin{gather*}
    x = u z, \quad y = v z, \quad z\in \NN^{d+1}, \quad u,v \in \NN;
    \\
    z_ 0\ldots z_d = \frac{x_0\ldots x_{d}}{u^{d+1}} \und{d}\ll \frac{2^{k_0+\ldots + k_{d}}}{u^{d+1}} \und{d}\ll \frac{p^d}{T u^{d+1}}, \quad
    z_0\ldots z_{d} \und{d}\ll \frac{p^d}{T v^{d+1}}, \quad z_0 \le H.
\end{gather*}
Therefore,
\begin{gather*}
    R_0 = \frac{1}{p^d} \sum_{k,n\in K(T)} \cA'(2^k,2^n) \und{d}\ll \frac{1}{p^d} \sum_{u,v\ge 1} \sum_{z_0\ldots z_{d} \ll p^d (T \max\{u^{d+1}, v^{d+1}\})^{-1}, \atop z_0\le H} 1 \ll{}
    \\
    {}\ll
    \frac{1}{p^d} \sum_{1\le v \le u} \sum_{z_0\ldots z_{d} \ll p^d (T u^{d+1})^{-1}, \atop z_0 \le H} 1
    \und{d}\ll  \frac{1}{p^d} \sum_{1\le v \le u} \frac{p^d}{T u^{d+1}} (\log p)^{d-1} \log H \ll \frac{(\log p)^{d-1} \log H}{T} \und{d}\ll \mu(T).
\end{gather*}
We took into account the condition $d\ge 2$. This completes the proof of the lemma.
\end{proof}

\begin{proof}[Proof of Theorem \ref{th3}]
Put
$$
    T = \frac{2^{d+1} (\log p)^{d-1} \log H}{\lambda}.
$$
Then
$$
    1\le x_0 < 2^{k_0+1} < H, \quad 1\le x_0 |x_1|\ldots |x_d| \le 2^{d+1 + k_0 + \ldots + k_d} \le \frac{p^d 2^{d+1}}{T} = \frac{\lambda p^d}{ (\log p)^{d-1} \log T}
$$
for all $(x_0,\ldots,x_d) \in \Omega = \Omega (T,R,p)$. Condition  \eqref{6.2} holds. Using \eqref{6.3} and Lemma \ref{l4.1}, we have
$$
    \frac{1}{p^d} \# \left\{a\in \ZZ_p^d: q_p(a,H) > \frac{\lambda}{(\log p)^{d-1} \log H}\right\} \le
    \frac{4 \sigma^2}{\mu^2} \und{d}\ll \frac{1}{\mu} \ll \frac{T}{(\log p)^{d-1} \log H} \und{d}\ll \frac{1}{\lambda}.
$$
This completes the proof of Theorem \ref{th3}.
\end{proof}


\end{document}